\documentclass[12pt,british]{article}

\usepackage{babel}
\usepackage{mathrsfs,diagbox} 
\usepackage{amsmath,amsfonts,amssymb,amsthm}
\usepackage{ucs} 
\usepackage[utf8x]{inputenc} 
\usepackage[T1]{fontenc} 
\usepackage{enumerate}
\usepackage{eucal}
\usepackage{comment}
\usepackage[noBBpl]{mathpazo}
\usepackage[matrix,arrow]{xy}
\usepackage{epsfig}

\usepackage{physics} 

\usepackage[pdftex,                %
implicit=true,
bookmarks=true,
breaklinks=true,
bookmarksnumbered = true,
colorlinks= true,
urlcolor= green,
anchorcolor = yellow,
citecolor=blue,
pdfborder= {0 0 0}
]{hyperref}

\usepackage{ytableau} 
\usepackage{multicol}

\usepackage{subcaption}

\makeatletter

\renewcommand{\p@enumii}{}
\def\@enum@{\list{\csname label\@enumctr\endcsname}%
{\usecounter{\@enumctr}\def\makelabel##1{
\normalfont\ignorespaces\emph{{##1}~}}
\setlength{\labelsep}{3pt}
\setlength{\parsep}{0pt}
\setlength{\itemsep}{0pt}
\setlength{\leftmargin}{0pt}
\setlength{\labelwidth}{0pt}
\setlength{\listparindent}{\parindent}
\setlength{\itemsep}{0pt}
\setlength{\itemindent}{0pt}
\topsep=3pt plus 1pt minus 1 pt}}
\makeatother

\renewcommand{\epsilon}{\ensuremath{\varepsilon}}
\renewcommand{\phi}{\ensuremath{\varphi}}

\renewcommand{\to}{\ensuremath{\longrightarrow}}

\newcommand{\Z}{\ensuremath{\mathbb Z}}

\newcommand{\St}[1][2]{\ensuremath{\mathbb S}^{#1}}

\newcommand{\aut}[1]{\ensuremath{\operatorname{\text{Aut}}\left({#1}\right)}}

\makeatletter
\def\@map#1#2[#3]{\mbox{$#1 \colon\thinspace #2 \to #3$}}
\def\map#1#2{\@ifnextchar [{\@map{#1}{#2}}{\@map{#1}{#2}[#2]}}
\makeatother

\newcommand{\ang}[1]{\ensuremath{\left\langle #1\right\rangle}}
\newcommand{\set}[2]{\ensuremath{\left\{#1 \,\mid\, #2\right\}}}
\newcommand{\setang}[2]{\ensuremath{\ang{#1 \,\mid\, #2}}}

\newtheoremstyle{theoremm}{}{}{\itshape}{}{\scshape}{.}{ }{}
\theoremstyle{theoremm}

\newtheorem{thm}{Theorem}
\newtheorem{lem}[thm]{Lemma}
\newtheorem{prop}[thm]{Proposition}

\newtheoremstyle{remark}{}{}{}{}{\scshape}{.}{ }{}
\theoremstyle{remark}

\newtheorem{rem}[thm]{Remark}
\newtheorem{rems}[thm]{Remarks}

\newtheoremstyle{comment}{}{}{\bfseries}{}{\bfseries}{:}{ }{}
\theoremstyle{comment}

\newcommand{\reth}[1]{Theorem~\protect\ref{th:#1}}
\newcommand{\relem}[1]{Lemma~\protect\ref{lem:#1}}

\newcommand{\resec}[1]{Section~\protect\ref{sec:#1}}
\newcommand{\resubsec}[1]{Subsection~\protect\ref{subsec:#1}}

\newcommand{\rerems}[1]{Remarks~\protect\ref{rems:#1}}
\newcommand{\req}[1]{equation~(\protect\ref{eq:#1})}
\newcommand{\reqref}[1]{(\protect\ref{eq:#1})}
\allowdisplaybreaks

\begin{document}

\title{The R$_{\infty}$-property for braid groups over orientable surfaces}

\author{KAREL~DEKIMPE\\
KU Leuven Campus Kulak Kortrijk,\\
Etienne Sabbelaan 53, 8500 Kortrijk, Belgium.\\
e-mail:~\texttt{karel.dekimpe@kuleuven.be}\vspace*{4mm}\\
DACIBERG~LIMA~GON\c{C}ALVES\\
Departamento de Matem\'atica - IME-USP,\\
Rua~do~Mat\~ao~1010~CEP:~05508-090  - S\~ao Paulo - SP - Brazil.\\
e-mail:~\texttt{dlgoncal@ime.usp.br}\vspace*{4mm}\\
OSCAR~OCAMPO~\\
Universidade Federal da Bahia,\\
Departamento de Matem\'atica - IME,\\
Av. Milton Santos~S/N~CEP:~40170-110 - Salvador - BA - Brazil.\\
e-mail:~\texttt{oscaro@ufba.br}
}

\date{\today}

\maketitle

\begin{abstract} 
Let $\Sigma_{g,p}$ be an orientable surface of genus $g$ and of finite type without boundary (i.e.\  an orientable closed surface with a finite number $p$ of points removed). 
In this paper we study the R$_{\infty}$-property for the surface pure braid groups $P_n(\Sigma_{g,p})$ as well as for the full surface braid groups $B_n(\Sigma_{g,p})$. We show that, with few exceptions, these groups have the R$_{\infty}$-property. 
\end{abstract}

\let\thefootnote\relax

\section{Introduction}\label{intro}

Consider a group $G$ and a fixed endomorphism $\phi$ of $G$. Two elements $x$ and $y$ of $G$ are said to be twisted conjugate (via $\phi$) if and only if there exists a $z\in G$ such that $x = z y \phi(z)^{-1}$. The relation of being twisted conjugate is easily seen to be an equivalence relation and the number of equivalence classes (also referred to as twisted conjugacy classes or Reidemeister classes) is called the Reidemeister number $R(\phi)$ of $\phi$. This Reidemeister number is either a positive integer or $\infty$.

These Reidemeister numbers appear naturally in algebraic topology and to be more precise in Nielsen--Reidemeister fixed point theory. Here one is interested in the number of fixed point classes of a selfmap $f$ of a space $X$. This number is called the Reidemeister number $R(f)$ of the map $f$, and one can show that $R(f)=R(f_\ast)$, where $f_\ast\colon \pi_1(X) \to \pi_1(X) $ is the induced endomorphism on the fundamental group $\pi_1(X)$ of $X$. 

A group $G$ is said to have the R$_\infty$-property in case $R(\phi)=\infty$ for all automorphisms $\phi \in \aut{G}$.
 The study of groups with that  property was initiated by Fel'shtyn and Hill \cite{FH} and since the beginning of this century there has been a growing interest in the study of groups having this R$_\infty$-property.

A non-exhaustive list of examples of groups of which we know that they have the R$_\infty$-property, are the non-elementary Gromov hyperbolic groups \cite{F,LL}, most of the Baumslag--Solitar groups \cite{FG1}  and groups quasi--isometric to Baumslag--Solitar groups \cite{TW},\
 generalized Baumslag--Solitar groups \cite{L}, many linear groups \cite{FN,N}, several families of lamplighter groups \cite{GW,T}, some spherical and affine Artin-Tits groups \cite{CS},  pure virtual twin groups \cite{NNS}, and virtual braid (twin) groups \cite{DGO2}.

Emil Artin introduced the braid groups of the $2$-disc  in 1925 and continued the study of them in 1947~\cite{A1,A2}. These groups have since then been referred to as Artin Braid groups. Zariski~\cite{Z} was the first to study braids on surfaces and this was later further extended  by Fox and Neuwirth to braid groups of arbitrary topological spaces by using configuration spaces as follows~\cite{FoN}. Let $M$ be a topological space, and let $n\in \mathbb N$. The \textit{$n$th ordered configuration space of $M$}, denoted by $F_{n}(M)$, is defined by:
\begin{equation*}
F_n(M)=\left\{(x_{1},\ldots,x_{n})\in M^{n}\,|\,x_{i}\neq x_{j}\,\, \text{if}\,\, i\neq j,\,i,j=1,\ldots,n\right\}.
\end{equation*}
The \textit{$n$-string pure braid group $P_n(M)$ of $M$} is defined by $P_n(M)=\pi_1(F_n(M))$. The symmetric group $S_{n}$ on $n$ letters acts freely on $F_{n}(M)$ by permuting coordinates, giving rise to the $n$th unordered configuration space
$F_n(M)/S_{n}$. The \textit{$n$-string braid group $B_n(M)$ of $M$} is then defined as $B_n(M)=\pi_1(F_n(M)/S_{n})$. This gives rise to the following short exact sequence:
\begin{equation}\label{eq:ses}
1 \to P_{n}(M) \to B_{n}(M) \stackrel{\sigma}{\longrightarrow}S_{n} \to 1.
\end{equation}

The map $\map{\sigma}{B_{n}(M)}[S_{n}]$ is the standard homomorphism that associates to any braid in $B_{n}(M)$ a permutation in $S_{n}$ and $\textrm{Ker}(\sigma)=P_n(M)$.

 When $M=D^2$ (the disc)  then $B_n(D^2)$ (resp.\ $P_n(D^2)$) is the classical Artin braid group denoted by $B_n$ (resp.\ the classical pure Artin braid group denoted by $P_n$).

The R$_{\infty}$-property was studied for Artin braid groups in \cite{FG} for the whole group $B_n$ and in \cite{DGO} for the pure subgroup $P_n$. 
Let $S_{g,p}$ be a surface of finite type, i.e. $S_{g,p}$ is a closed surface of genus $g$ (possibly non-orientable) with a finite number ($p \geq 0$) of points removed. 
After having obtained the results for the Artin braid groups, it is now a natural question to study the R$_{\infty}$-property for the surface braid groups (resp.\ surface pure braid groups) $B_n(S_{g,p})$ (resp.\ $P_n(S_{g,p})$). 
For the case where $n=1$ we have that $P_1(S_{g,p})=B_1(S_{g,p})=\pi_1(S_{g,p})$ and here
the result is well known, and the information (in the orientable case) is given in the tables below. 
So, from now on, we will assume that $n\geq 2$ unless it is explicitly  stated otherwise.

In this paper we will study the $R_\infty$ property only for the case of orientable surfaces. To do this, we divide the orientable surfaces of finite type into three families.
\begin{itemize}
	\item[$\mathcal{F}_1$:] The punctured sphere  $\mathbb{S}^2$ with $p$ points removed for $p=0,1,2$.  

\item[$\mathcal{F}_2$:] 
\begin{itemize}
	\item[a)] Orientable closed surfaces different from $\mathbb{S}^2$, $T^2$. 
	\item[b)] Orientable punctured surfaces $\Sigma_{g,p}$ where $g$ is the genus and $p$ is the number of punctures in the closed surface $\Sigma_g$, for: 
\begin{itemize}
	\item[i)] $g=0$ and $p\geq 3$,
	\item[ii)] $g=1$ and $p\geq 2$,
	\item[iii)] $g\geq 2$ and $p\geq 1$. 
\end{itemize}
	\end{itemize}
	
\item[$ \mathcal{F}_3$:] The torus $\Sigma_{1,0}=T^2$ and $\Sigma_{1,1}$ the torus minus one point.

\end{itemize}

  In the following table we record the information that we know until now about the R$_{\infty}$-property for the surface braid groups, 
$P_n(\Sigma_{g,p})$, $B_n(\Sigma_{g,p})$.

\begin{table}[htb] 
\centering
\begin{tabular}{c c c c c c c}
\hline
   & &  $ $ & &  & & \vspace*{-0.3cm}\\
\textcolor{blue}{Family} & & \textcolor{blue}{$\pi_1$ has R$_{\infty}$} & & \textcolor{blue}{$P_n(\Sigma_{g,p})$ has R$_{\infty}$} & & \textcolor{blue}{$B_n(\Sigma_{g,p})$ has R$_{\infty}$}  \\[0.2cm] \hline \hline 
   & &  $ $ & &  & & \vspace*{-0.3cm}\\
$\mathcal{F}_1$ & & No & & Unknown   & &  Unknown  \\ 
 & &  & & except, yes  for $S=\Sigma_{0,2}$  & &  except, yes  for $S=\Sigma_{0,2}$ \\ 
 & &  & & and $n\geq 2$, and yes for  & &  and $n\geq 2$, and yes for  \\ 
\vspace{0.2cm}
 & &  & & $S=\Sigma_{0,1}$ iff $n\geq 3$  & &  $S=\Sigma_{0,1}$ iff $n\geq 3$ \\ 
$\mathcal{F}_2$ & & Yes & & Unknown & &  Unknown   \\
\vspace{0.2cm}
 & &  & &  & &  except, yes  for $S=\Sigma_{0,3}$  \\
$\mathcal{F}_3$ & & No for $T^2$ & & Unknown & & Unknown \\ 
                & & Yes for $\Sigma_{1,1}$ & &  & &  \vspace{0.2cm}\\ 
\hline
\end{tabular}
\caption{The R$_{\infty}$-property for $P_n(\Sigma_{g,p})$ and $B_n(\Sigma_{g,p})$ before this paper.}
\label{tab:info1}
\end{table}

\begin{rem} The exceptional cases which appear in the table above come from the fact that
for $n\geq 2$, the groups $P_n(\Sigma_{0,1})$, $P_n(\Sigma_{0,2})$, $B_n(\Sigma_{0,1})$, $B_n(\Sigma_{0,2})$ and $B_n(\Sigma_{0,3})$ have the R$_{\infty}$-property (see \cite[Theorem~1]{CS}), since there are isomorphisms among some surface braid groups and Artin-Tits groups: $P_n(\Sigma_{0,1}) \cong P(A_{n-1})$, $P_n(\Sigma_{0,2}) \cong P(B_n)$, $B_n(\Sigma_{0,1}) \cong A(A_{n-1})$, $B_n(\Sigma_{0,2})\cong A(B_n)$ and $B_n(\Sigma_{0,3})\cong A(\widetilde{C}_n)$. 
In the case of $\Sigma_{0,1}$ the result was first demonstrated in \cite{FG} for the Artin braid group and in \cite{DGO} for the pure Artin braid group.
\end{rem}

The reason for dividing the surfaces into these three families is because we need different techniques to deal with the surfaces of  family $\mathcal{F}_1$ and those of family $\mathcal{F}_2$. The paper does not contain new results on the two braid groups of the two surfaces of family $\mathcal{F}_3$  as this is still work in progress.

The main results of this paper are formulated below. 
\begin{thm}\label{th:purerinfty}   
Let $\Sigma_{g,p}$ be a finite type surface which  belongs to $\mathcal{F}_1\cup \mathcal{F}_2$. 
The surface pure braid group $P_n(\Sigma_{g,p})$ has the R$_{\infty}$-property if and only if one of the statements below holds:
\begin{enumerate}
\item $\Sigma_{g,p}$ belongs to  $\mathcal{F}_2$ and $n\geq 1$,
\item $\Sigma_{g,p} =\Sigma_{0,0}= \mathbb{S}^2$ and $n\geq 4$,
\item  $\Sigma_{g,p} =\Sigma_{0,1}=\mathbb{S}^2\setminus\{x_1\}$ and $n\geq 3$,
\item  $\Sigma_{g,p} =\Sigma_{0,2}=\mathbb{S}^2\setminus\{x_1, x_2\}$ and $n\geq 2$.
\end{enumerate}
\end{thm}

In order to prove the result for the whole group $B_n(\Sigma_{g,p})$, stated in the next theorem, we shall use \reth{purerinfty} and the following useful result: for all surfaces $S_{g,p}$ (orientable or not), $P_n(S_{g,p})$ is characteristic in $B_n(S_{g,p})$ with  one exception,  which is  when $S_{g,p}=\Sigma_{0,2}$ and $n=2$, see \cite[Theorem~1.5]{A}.

\begin{thm}\label{th:totalrinfty}  
Let $\Sigma_{g,p}$ be a surface which belongs to  $\mathcal{F}_1\cup \mathcal{F}_2$. 
The braid group $B_n(\Sigma_{g,p})$ has the R$_{\infty}$-property, if and only if $P_n(\Sigma_{g,p})$ has the  R$_{\infty}$-property.
\end{thm}

 The following table summarises the information obtained in this work as well the status of the question studied here for  a finite type surface orientable surface $\Sigma_{g,p}$.

\begin{table}[htb] 
\centering
\begin{tabular}{c c c c c c c}
\hline
     & &  $ $ & &  & & \vspace*{-0.3cm}\\
\textcolor{blue}{Family} & & \textcolor{blue}{$\pi_1$ has R$_{\infty}$} & & \textcolor{blue}{$P_n(\Sigma_{g,p})$ has R$_{\infty}$} & & \textcolor{blue}{$B_n(\Sigma_{g,p})$ has R$_{\infty}$}  \\[0.2cm] \hline \hline 
     & &  $ $ & &  & & \vspace*{-0.3cm}\\
$\mathcal{F}_1$ & & No & & Yes for most;  & & Yes for most;  \\ 
 & &  & & No for few cases  & & No for few cases  \\[0.2cm] 
\vspace{0.2cm}
$\mathcal{F}_2$ & & Yes & & Yes & & Yes \\
$\mathcal{F}_3$ & & No for $T^2$ & & Unknown & & Unknown \\ 
                & & Yes for $\Sigma_{1,1}$ & &  & & \vspace{0.2cm} \\   
\hline
\end{tabular}
\caption{The R$_{\infty}$-property for $P_n(\Sigma_{g,p})$ and $B_n(\Sigma_{g,p})$ after this paper.}
\label{tab:info2}
\end{table}

\begin{rem} In the table \ref{tab:info2}, for the family $\mathcal{F}_1$ the cases where $P_n(\Sigma_{g,p})$ ($n\geq 2$) does not have the R$_{\infty}$-property are precisely the cases $\mathbb{S}^2$ for $n=2,3$ and $S=D^2$ for $n=2$. The same holds for $B_n(\Sigma_{g,p})$.  
\end{rem}

This paper is organised as follows. 
In \resec{goldberg} we show that for any finite type surface $S_{g,p}$ (orientable or not) of genus $g\geq 0$ with $p\geq 0$ points removed, there is a short exact sequence
  \begin{equation}\label{eq:sesgintro}
	1\to  N\to P_n(S_{g,p})\to \Pi_{i=1}^n(\pi_1(S_{g,p})) \to 1
	\end{equation}
 where  $N$ is the normal closure of the Artin pure braid group in $P_n(S_{g,p})$. Then, in \resec{auto} we prove that the sequence \reqref{sesgintro} is characteristic for the surfaces of the family ${\mathcal F}_2$, being different from the sphere minus three points. 
In \resec{main} we prove \reth{purerinfty} and \reth{totalrinfty}, these are the main results of the paper about the R$_{\infty}$-property for orientable surface braid groups.

\subsection*{Acknowledgments}

The first author was supported by Methusalem grant METH/21/03 -- long term structural funding of the Flemish Government.
The second author was partially supported by the National Council for Scientific and Technological Development - CNPq through a \textit{Bolsa de Produtividade} 305223/2022-4  and by   Projeto Tem\'atico-FAPESP Topologia Alg\'ebrica, Geom\'etrica e Diferencial 2022/16455-6 (Brazil).
The third author was partially supported by the National Council for Scientific and Technological Development - CNPq through a \textit{Bolsa de Produtividade} 305422/2022-7.

\section{Goldberg's short exact sequence for pure braid groups over punctured surfaces (orientable or not)}\label{sec:goldberg}

	Let $S_{g,p}$ be a closed surface (orientable or not) of genus $g\geq 0$ with $p\geq 0$ punctures. Let 
    $D\subset S_{g,p}$ be a subset which is homeomorphic to the open disc of radius $1$ of the plane. Denote by 
     $D \stackrel{i}{\hookrightarrow} S_{g,p}$  the inclusion, and let $(z_1, \cdots, z_n)$ be a base point of $F_n(D)$ and of $F_n(S_{g,p})$.     This inclusion induces a morphism $i_\#\colon P_n(D) \to P_n(\Sigma_{g,p})$.
     Making use of such an embedding, 
in this section we prove that the pure braid groups over punctured surfaces fit into a short exact sequence of Goldberg's type \cite{G}, see Theorem~\ref{goldb}.
The proof is algebraic and in order to do that we shall use the presentations of $P_n(S_{g,p})$ given in \cite[Theorem~1]{La} for the case of the punctured sphere, in \cite[Theorem~5.1]{B1} for the punctured connected sum of tori and in \cite[Theorem~4.7]{D} for the case of the punctured connected sum of projective planes. 
In the first three subsections we discuss in more details these presentations,  where we also indicate how to identify $i_\#(P_n(D))$ in each of the surface braid groups, and in the last subsection we deal with the short exact sequence of Goldberg's type.

\subsection{The punctured sphere}\label{subsec:puncturedsphere}

The once-punctured sphere is homeomorphic to the (open) disc, see Figure~\ref{puncturedsphere} (a), and in this case we get the classical Artin pure braid groups \cite{A1}.

\begin{figure}[!htb]
    \centering
\begin{subfigure}[t]{0.5\textwidth}
        \centering
\includegraphics[scale=0.3]{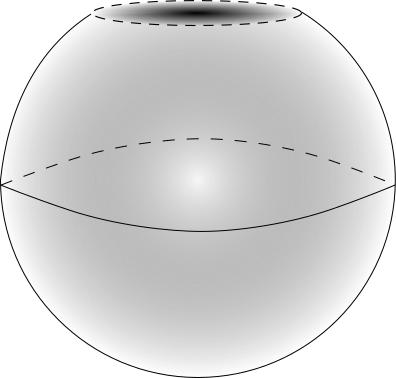}
\caption{Once-punctured.}
\label{puncturedspherea}
\end{subfigure}%
    ~ 
\begin{subfigure}[t]{0.5\textwidth}
        \centering
\includegraphics[scale=0.3]{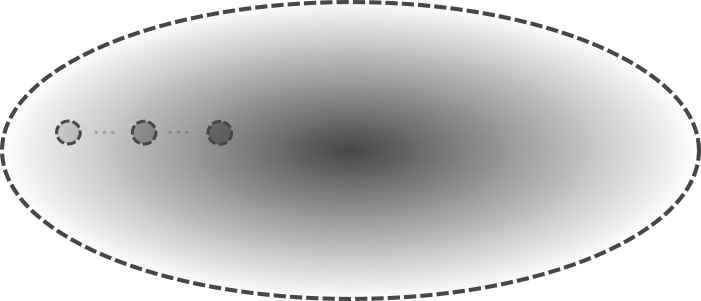}
\caption{With $p\geq 2$ punctures.}
\label{puncturedsphereb}
    \end{subfigure}%
\caption{The punctured sphere.}
\label{puncturedsphere}
\end{figure}

Let $p\geq 2$. 
Using the notation of \cite{La} and considering the $p$-punctured sphere as being the $(p-1)$-punctured disc (see Figure~\ref{puncturedsphere} (b)), one can easily see that $P_n(\Sigma_{0,p}) = P_{p-1,n}$ (so $m=p-1$).
Then \cite[Theorem~1]{La} provides a presentation of $P_n(\Sigma_{0,p}) = P_{p-1,n}$,  with a set of generators  
$$
\set{A_{i,j}}{1\leq i \leq p+n-2,\, p\leq j\leq p+n-1 \textrm{ and } i<j }
$$ 
subject to the following relations:
\begin{itemize}
\item[(P1)] $A_{i,j}^{-1}A_{r,s}A_{i,j}=A_{r,s}$ if ($i<j<r<s$) or ($r<i<j<s$).

\item[(P2)] $A_{i,j}^{-1}A_{j,s}A_{i,j}=A_{i,s}A_{j,s}A_{i,s}^{-1}$ if ($i<j<s$).

\item[(P3)] $A_{i,j}^{-1}A_{i,s}A_{i,j} = A_{i,s}A_{j,s}A_{i,s}A_{j,s}^{-1}A_{i,s}^{-1}$ if ($i<j<s$).

\item[(P4)] $A_{i,j}^{-1}A_{r,s}A_{i,j} = A_{i,s}A_{j,s}A_{i,s}^{-1}A_{j,s}^{-1}A_{r,s}A_{j,s}A_{i,s}A_{j,s}^{-1}A_{i,s}^{-1}$ if ($i<r<j<s$).
\end{itemize}

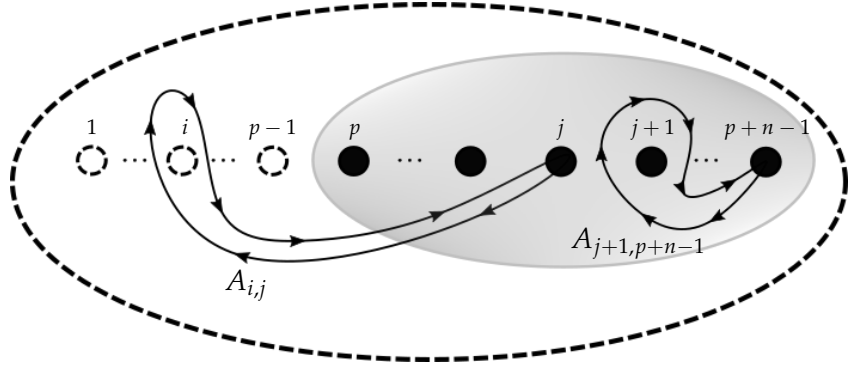
\begin{figure}[!htb]
\centering
{
\begin{picture}(300,140)

\put(0,0){\includegraphics[scale=0.6]{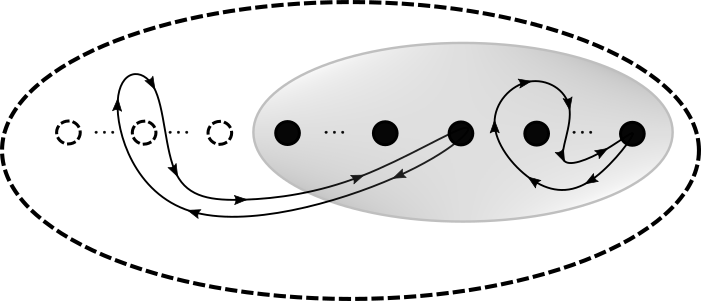}}
\put(29,86){{\scriptsize $1$}}
\put(65,86){{\scriptsize $i$}}
\put(90,86){{\scriptsize $p-1$}}
\put(128,86){{\scriptsize $p$}}
\put(207,86){{\scriptsize $j$}}
\put(235,86){{\scriptsize $j+1$}}
\put(270,86){{\scriptsize $p+n-1$}}
\put(81,27){{\small $A_{i,j}$}}
\put(212,43){{\small $A_{j+1,p+n-1}$}}
\end{picture}
}

\caption{Generator $A_{i,j}$ for $1\leq i \leq p+n-2$, $p\leq j\leq p+n-1$  and  $i<j$.
}
\label{aijpuncturedsphere}
\end{figure}

\begin{rem}
    In Figure~\ref{aijpuncturedsphere} we illustrate geometrically the generators $A_{i,j}$ of the group $P_n(\Sigma_{0,p})$, for $1\leq i \leq p+n-2$, $p\leq j\leq p+n-1$  and  $i<j$. 
We note that the set $\set{A_{i,j}}{p\leq i<j\leq p+n-1}$ corresponds to the set of Artin generators inside $P_n(\Sigma_{0,p})$.
\end{rem}

\subsection{The punctured connected sum of tori}\label{subsec:puncturedorientable}

Let $p\geq 1$ and $g\geq 1$. 
We shall use the presentation of $P_n(\Sigma_{g,p})$, the pure braid group of the $p$-punctured connected sum of $g$ tori $\Sigma_{g,p}$, as given in \cite[Theorem~5.1]{B1}. We note that this presentation had a few misprints which were corrected in \cite[Theorem~12]{BGG} and after private communication with P.\ Bellingeri (\cite{B2}) we fixed one more typo here.  A set of generators of $P_n(\Sigma_{g,p})$ given in \cite[Theorem~5.1]{B1} is 
$$
\set{A_{i,j}}{1\leq i\leq 2g+p+n-2,\, 2g+p\leq j\leq 2g+p+n-1,\, i<j}
$$ 
subject to the relations
\begin{itemize}
\item[(PR1)] $A_{i,j}^{-1}A_{r,s}A_{i,j}=A_{r,s}$ if ($i<j<r<s$) or ($r+1<i<j<s$), or ($i=r+1<j<s$ for even $r<2g$ or $r\geq 2g$).

\item[(PR2)] $A_{i,j}^{-1}A_{j,s}A_{i,j}=A_{i,s}A_{j,s}A_{i,s}^{-1}$ if ($i<j<s$).

\item[(PR3)] $A_{i,j}^{-1}A_{i,s}A_{i,j} = A_{i,s}A_{j,s}A_{i,s}A_{j,s}^{-1}A_{i,s}^{-1}$ if ($i<j<s$).

\item[(PR4)] $A_{i,j}^{-1}A_{r,s}A_{i,j} = A_{i,s}A_{j,s}A_{i,s}^{-1}A_{j,s}^{-1}A_{r,s}A_{j,s}A_{i,s}A_{j,s}^{-1}A_{i,s}^{-1}$ if ($i+1<r<j<s$) or \linebreak 
($i+1=r<j<s$ for odd $r<2g$ or $r>2g$).
\item[(ER1)] $A_{r+1,j}^{-1}A_{r,s}A_{r+1,j}=A_{r,s}A_{r+1,s}A_{j,s}^{-1}A_{r+1,s}^{-1}$ if $r$ odd, $r<2g$ and $r+1<j<s$.

\item[(ER2)] $A_{r-1,j}^{-1}A_{r,s}A_{r-1,j} = A_{r-1,s}A_{j,s}A_{r-1,s}^{-1}A_{r,s}A_{j,s}A_{r-1,s}A_{j,s}^{-1}A_{r-1,s}^{-1}$ if $r$ even,  $r\leq 2g$ and $r-1<j<s$. 

\end{itemize}

\begin{figure}[!htb]
\centering
\includegraphics[scale=0.19]{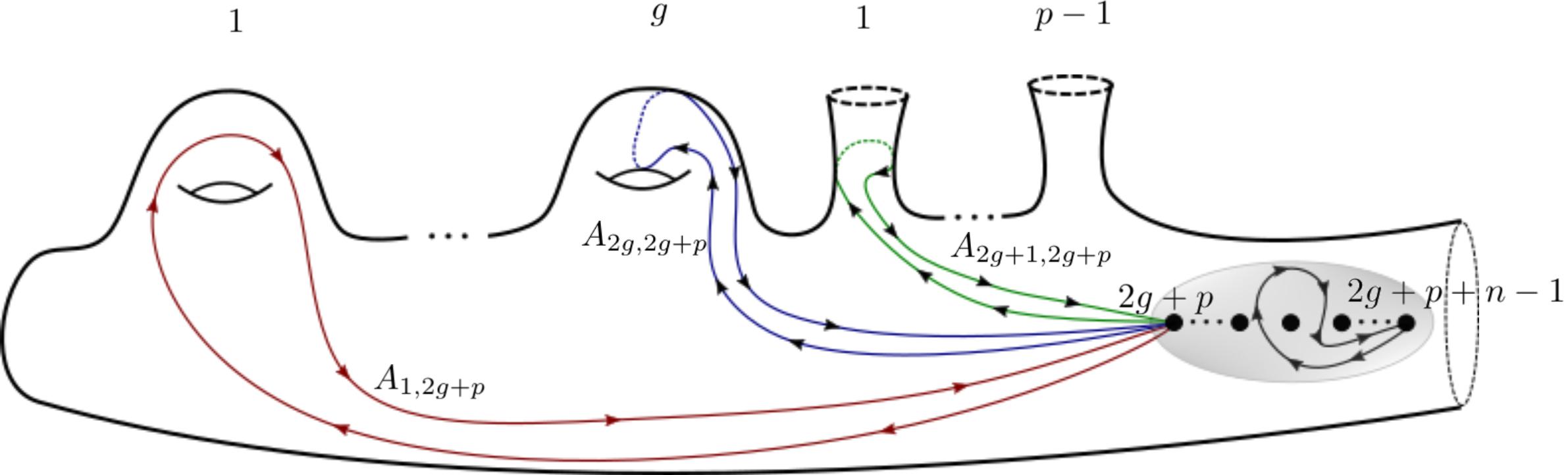}
\caption{Generator $A_{i,j}$ for $1\leq i\leq 2g+p+n-2$, $2g+p\leq j\leq 2g+p+n-1$ and  $i<j$.}
\label{aijorientable}
\end{figure}

\begin{rem}
    In Figure~\ref{aijorientable} we illustrate geometrically the generators $A_{i,j}$ of the group $P_n(\Sigma_{g,p})$, for $1\leq i\leq 2g+p+n-2$, $2g+p\leq j\leq 2g+p+n-1$  and  $i<j$. 
We note that the set $\set{A_{i,j}}{2g+p\leq i<j\leq 2g+p+n-1}$ corresponds to the set of Artin generators inside $P_n(\Sigma_{g,p})$.
\end{rem}

\subsection{The punctured non orientable surfaces}\label{subsec:puncturednonorientable}

For this case we use the presentation given in \cite[Theorem~4.7]{D}. Let $g,p\geq 1$.  A set of generators of $P_n(N_{g,p})$ is given by $A_{i,j}$ and $\rho_{r,k}$ for $1\leq i<j$, $p+1\leq j,\, r\leq p+n$ and $1\leq k\leq g$ subject to the relations

\begin{itemize}
\item[(a)] The ``Artin-type relations''
$$
A_{r,s}A_{i,j}A_{r,s}^{-1} = \left\{
\begin{array}{ll}
A_{i,j} & \mbox{if } (i<r<s<j) \textrm{ or } \\
 & \hfill  (r<s<i<j) \\
A_{s,j}^{-1}A_{i,j}A_{s,j} & \mbox{if } i=r<s<j \\
A_{i,j}^{-1}A_{r,j}^{-1}A_{i,j}A_{r,j}A_{i,j} & \mbox{if } r<i=s<j \\
A_{s,j}^{-1}A_{r,j}^{-1}A_{s,j}A_{r,j}A_{i,j}A_{r,j}^{-1}A_{s,j}^{-1}A_{r,j}A_{s,j} & \mbox{if } r<i<s<j. \\
\end{array} 
\right.
$$

\item[(b)] For every $p+1\leq i<j\leq p+n$ and $1\leq k,\, l\leq g$
$$
\rho_{i,k}\rho_{j,l}\rho_{i,k}^{-1} = \left\{
\begin{array}{ll}
\rho_{j,l} & \mbox{if } k<l \\
\rho_{j,k}^{-1}A_{i,j}^{-1}\rho_{j,k}^{2} & \mbox{if } k=l \\
\rho_{j,k}^{-1}A_{i,j}^{-1}\rho_{j,k}A_{i,j}^{-1}\rho_{j,l}A_{i,j}\rho_{j,k}^{-1}A_{i,j}\rho_{j,k} & \mbox{if } k>l. \\
\end{array} 
\right.
$$

\item[(c)] The surface relation, for every $p+1\leq j\leq p+n$
$$
\prod_{l=1}^{g}\rho_{j,l}^{2}=\left( \prod_{i=1}^{j-1}A_{i,j} \right)\left( \prod_{s=1+j}^{p+n}A_{j,s} \right).
$$
Note that when $j=p+n$ then in the right-hand side of the equality the second factor disappears. 

\item[(d)] For every $1\leq i<j$, $p+1\leq j,\, k\leq p+n$, $k\neq j$ and $1\leq l\leq g$
$$
\rho_{k,l}A_{i,j}\rho_{k,l}^{-1} = \left\{
\begin{array}{ll}
A_{i,j} & \mbox{if } k<i \textrm{ or } j<k \\
\rho_{j,l}^{-1}A_{i,j}^{-1}\rho_{j,l} & \mbox{if } k=i \\
\rho_{j,l}^{-1}A_{k,j}^{-1}\rho_{j,l}A_{k,j}^{-1}A_{i,j}A_{k,j}\rho_{j,l}^{-1}A_{k,j}\rho_{j,l} & \mbox{if } i<k<j. \\
\end{array} 
\right.
$$

\end{itemize}

\begin{figure}[!htb]
\centering
\includegraphics[scale=0.15]{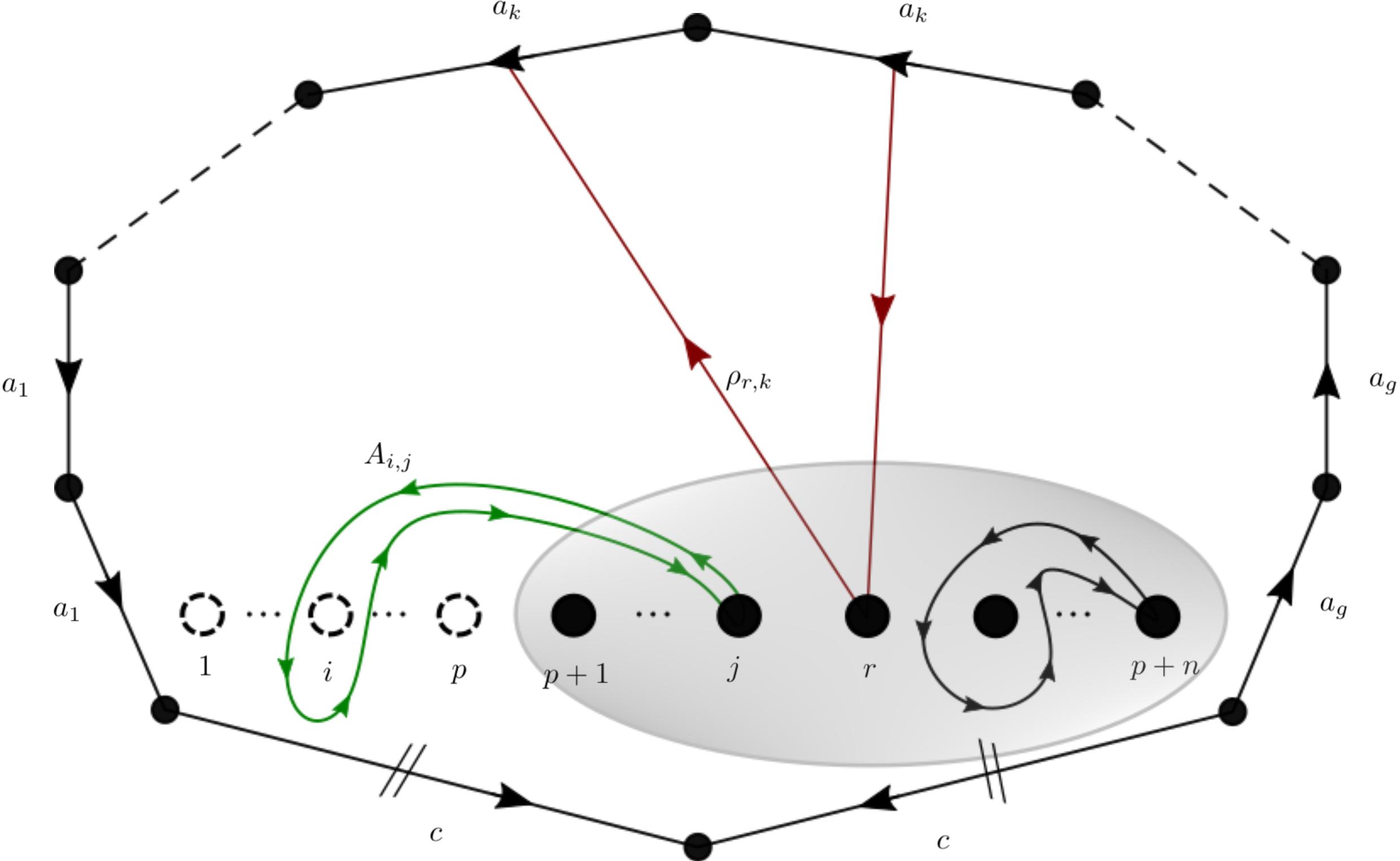}
\caption{Generators $A_{i,j}$ and $\rho_{r,k}$ for $1\leq i<j$, $p+1\leq j,\, r\leq p+n$ and $1\leq k\leq g$.}
\label{aijnonorientable}
\end{figure}

\begin{rem}
    In Figure~\ref{aijnonorientable} we illustrate geometrically the generators $A_{i,j}$ and $\rho_{r,k}$ of the group $P_n(N_{g,p})$, for $1\leq i<j$, $p+1\leq j,\, r\leq p+n$ and $1\leq k\leq g$. 
We note that the set $\set{A_{i,j}}{p+1\leq i<j\leq p+n}$ corresponds to the set of Artin generators inside $P_n(N_{g,p})$.
\end{rem}

\subsection{Goldberg's short exact sequence for surface pure braid groups}

The short exact sequence of the following theorem is known for closed surfaces, see \cite{G} for $S_{g,p}\ne \mathbb{S}^2, \mathbb{R}P^2$ and \cite{GG3} for $S=\mathbb{R}P^2$. We extend it here to the case of punctured surfaces.

	\begin{thm}\label{goldb} 
	Let $S_{g,p}$ be a closed surface (orientable or not) of genus $g\geq 0$ with $p\geq 0$ points removed. Let $N$ denote the normal subgroup  of $P_n(S_{g,p})$ generated by the image of the Artin pure braids via the inclusion of the disc $D$ into $S_{g,p}$, $i\colon D \hookrightarrow S_{g,p}$. 
	Then 
  \begin{equation}\label{eq:sesg}
	1\to  N\to P_n(S_{g,p})\to \Pi_{i=1}^n(\pi_1(S_{g,p})) \to 1
	\end{equation}
	is a short exact sequence.
  \end{thm}

\begin{proof}

For $\Sigma_{0,0}=\mathbb{S}^2$ this is obvious since $P_n(\St[2])$ is generated by the images of the pure Artin braids.  For $S_{g,0}\ne \mathbb{S}^2, \mathbb{R}P^2$ this was proved by Goldberg \cite{G} (and this was in the orientable case a  conjecture of Birman \cite{Bi}). For the case $N_{1,0}=\mathbb{R}P^2$ see Gon\c calves and Guaschi \cite{GG3}.

For the remaining cases (of punctured surfaces) the proof is algebraic by considering the presentations for the given groups given in the subsections above. Let $p\geq 1$.

\begin{itemize}
\item[Case 1.] First we prove the result for the punctured sphere. 
We consider in this case $p\geq 2$ since the pure braid group of the 1-punctured sphere $P_n(\Sigma_{0,1})$ is exactly  the pure braid group of the disc $P_n(D)$ and then the result is trivial. 

Let $p\geq 2$. We shall use the presentation of $P_n(\Sigma_{0,p})$ given in \resubsec{puncturedsphere}. 
We note that, in this presentation of $P_n(\Sigma_{0,p})$, the generators coming from the Artin generators of $P_n(D)$, via the inclusion $i\colon D \hookrightarrow \Sigma_{0,p}$, are the elements $A_{i,j}$ for $p\leq i<j\leq p+n-1$. 
Now, we will describe a presentation of the quotient of $P_n(\Sigma_{0,p})$ by the normal closure  subgroup generated by the Artin pure braids $A_{i,j}$, with $p\leq i<j\leq p+n-1$, that we called $N$.

Let us add the relations $A_{i,j}=1$ for $p\leq i< j\leq p+n-1$ to the presentation of $P_n(\Sigma_{0,p})$ given in \resubsec{puncturedsphere}. This implies that
for the quotient group $P_n(\Sigma_{0,p})/N$ we can take  $\set{A_{i,j}}{ 1\leq i\leq p-1 \ , \ p\leq j\leq p+n-1 }$ as the set of generators. 
The relations (P2) become trivial in the quotient and the relations (P1), (P3) and (P4) lead to the following 
relations respectively
\begin{itemize}
\item[(QR1)] $A_{i,j}^{-1}A_{r,s}A_{i,j}=A_{r,s}$ if  
($r<i<j<s$).

\item[(QR2)] $A_{i,j}^{-1}A_{i,s}A_{i,j} = A_{i,s}$ if ($i<j<s$).

\item[(QR3)] $A_{i,j}^{-1}A_{r,s}A_{i,j} = 
A_{r,s}
$ if ($i<r<j<s$).
\end{itemize}   

We can actually collect the relations (QR1), (QR2) and (QR3) together to obtain the relations
\begin{itemize}
\item[(QR)] $ [A_{i,j}, A_{r,s}]=1$ if ($1\leq r,i\leq p-1$) and ($p\leq j < s \leq p+n-1$).
\end{itemize}
This shows that for every $\ell\in \{ p,\ldots,p+n-1 \}$ the group $T_{\ell}$ generated by the set $\set{A_{k,\ell}}{1\leq k\leq p-1}$ is a free group of rank $p-1$ and
$$
P_n(\Sigma_{0,p})/N \cong T_1\oplus T_2\oplus \cdots \oplus T_n \cong \left( \pi_1(\Sigma_{0,p}) \right)^n.
$$
This concludes the proof for the punctured sphere.

\item[Case 2.] Now we prove this result for the case of the $p$-punctured connected sum of $g$ tori $\Sigma_{g,p}$. We shall use the presentation of the pure braid group $P_n(\Sigma_{g,p})$ given in \resubsec{puncturedorientable}.  
We note that the generators coming from the Artin generators of $P_n(D)$, via the inclusion $i\colon D \hookrightarrow \Sigma_{g,p}$, are the elements $A_{i,j}$ for $2g+p\leq i<j\leq 2g+p+n-1$. 
Now we add the relations $A_{i,j}=1$ for $2g+p\leq i<j\leq 2g+p+n-1$ to the presentation given in \resubsec{puncturedorientable} to deduce a presentation of the quotient group $P_n(\Sigma_{g,p})/N$.
We claim that the set $\set{A_{i,j}}{1\leq i\leq 2g+p-1,\, 2g+p\leq j\leq 2g+p+n-1}$ constitutes a set of generators for the group $P_n(\Sigma_{g,p})/N$ which are subject to the relations $[A_{i,j}, A_{r,s}]=1$ for $1\leq i,\, r\leq 2g+p-1$ and $2g+p\leq j,\, s\leq 2g+p+n-1$ with $j\neq s$. 
The claim is an immediate consequence of the following:
\begin{itemize}
	\item[(QR1)]  $[A_{i,j}, A_{r,s}]= 1$ if ($r+1<i<j<s$) or ($i=r+1<j<s$ for any $1\leq r\leq 2g+p-2$).
 \newline
 This follows directly from (PR1) when $r+1<i<j<s$ and from (PR1) in case $i=r+1<j<s$ for even $r<2g$ or $r\geq 2g$ and from (ER1) when $i=r+1<j<s$ for odd $r<2g$.

	\item[(QR2)] $[A_{i,j}, A_{i,s}]= 1$ if $i<j<s$.
	
	In fact, since $A_{j,s}=1$ for $2g+p\leq j<s \leq 2g+p+n-1$ then
	from (PR3):
	$$
	A_{i,j}^{-1}A_{i,s}A_{i,j}=A_{i,s}A_{j,s}A_{i,s}A_{j,s}^{-1}A_{i,s}^{-1}
	$$
	we obtain the relation (QR2).
	
	\item[(QR3)] $[A_{i,j}, A_{r,s}]= 1$ if ($i+1<r<j<s$) or ($i+1=r<j<s$ for any $2\leq r\leq 2g+p-2$).
	\newline
	Using once more time the fact that $A_{j,s}=1$ for $2g+p\leq j<s \leq 2g+p+n-1$ and from (PR4):
	$$
	A_{i,j}^{-1}A_{r,s}A_{i,j}=	A_{i,s}A_{j,s}A_{i,s}^{-1}A_{j,s}^{-1}A_{r,s}A_{j,s}A_{i,s} A_{j,s}^{-1}A_{i,s}^{-1}
	$$
	we get the relation (QR3) when ($i+1<r<j<s$) or ($i+1=r<j<s$ for odd $r<2g$ or $r> 2g$). When $r\leq 2g$ is even the verification is similar using (ER2).

\end{itemize}
Note that all of the relations (PR2) became trivial in the quotient group $P_n(\Sigma_{g,p})/N$. 

Hence the quotient group $P_n(\Sigma_{g,p})/N$ is isomorphic to  $\Pi_{i=1}^n(\pi_1(\Sigma_{g,p}))$, the direct product of free groups of rank $2g+p-1$. 

\item[Case 3.] Finally we consider the case of the punctured connected sum of projective planes.
For this case we use the presentation given in \resubsec{puncturednonorientable}. 
The generators coming from the Artin generators of $P_n(D)$, via the inclusion $i\colon D \hookrightarrow N_{g,p}$, are the elements $A_{i,j}$ for $p+1\leq i<j\leq p+n$.  
We add the relations $A_{i,j}=1$ for $p+1\leq i<j\leq p+n$ to the presentation given in \resubsec{puncturednonorientable} to deduce the following presentation of the quotient group $P_n(N_{g,p})/N$. 

The set of generators is given by $A_{i,j}$ for $1\leq i\leq p$ and $p+1\leq j\leq p+n$ and $\rho_{r,k}$ for $p+1\leq r\leq p+n$ and $1\leq k\leq g$ subject to the relations
\begin{itemize}
\item[(QR1)] $[A_{i,j}, A_{r,s}]=1$ for ($i<r<s<j$) or ($i=r<s<j$) or ($r<i<s<j$).

This follows from item (a).

\item[(QR2)] $[\rho_{r,k}, \rho_{j,l}]=1$ for $p+1\leq r<j\leq p+n$ and $1\leq k,l\leq g$.

This follows from item (b). 

\item[(QR3)] $[\rho_{k,l}, A_{i,j}]=1$ for $j\neq k$.

This relation follows from (d).

\item[(QSR)] From item (c) we have that for every $p+1\leq j\leq p+n$
$$
\prod_{l=1}^{g}\rho_{j,l}^{2}= \prod_{i=1}^{p}A_{i,j}.
$$
\end{itemize}

Recall that a presentation for the fundamental group of the punctured connected sum of projective planes, $\pi_1(N_{g,p})$, is given by 
$$
\setang{A_{i}, \rho_{l} \textrm{ for } 1\leq i\leq p \textrm{ and } 1\leq l\leq g}{\prod_{l=1}^{g}\rho_{l}^{2}= \prod_{i=1}^{p}A_{i}}.
$$
Therefore, the quotient group $P_n(N_{g,p})/N$ is isomorphic to the direct product of $n$ copies of $\pi_1(N_{g,p})$ (one for each $j\in \{p+1, \ldots, p+n\}$).

\end{itemize}

\end{proof}

\section{Automorphisms of $P_n(S)$ for $S$ a surface of finite type}\label{sec:auto}

Let $\Sigma_{g,p}$ be  an arbitrary finite surface of genus $g$ with $p\geq 0$ points removed. Let $M^{\ast}_n(\Sigma_{g,p})$ denote the \emph{extended mapping class group}, defined to be the group of isotopy classes of (possibly orientation-reversing) homeomorphisms of 
$(\Sigma_{g,p},z)$,  where  $z=\{ z_1,\ldots,z_n \}$ is a set of $n$ distinct points in $\Sigma_{g,p}$ and $(z_1,\ldots,z_n)$ is  a base point in the configuration space $F_n(\Sigma_{g,p})$.
Let  $h\colon \Sigma_{g,p} \to \Sigma_{g,p}$ be a homeomorphism which leaves $z$ invariant,  i.e.  $h(z)=z$, so the isotopy class of $h$, denoted by  $[h]$ is an element of $M^{\ast}_n(\Sigma_{g,p})$. The
 homeomorphism $h$ defines a permutation  $\sigma_{h}\in S_n$ given by the following equation:
$$h(z_i)=z_{\sigma_{h}(i)}.$$
We will consider the morphisms $\Phi\colon M^{\ast}_n(\Sigma_{g,p}) \to  \aut{B_n(\Sigma_{g,p})}$,  $\Phi_0\colon M^{\ast}_n(\Sigma_{g,p}) \to  \aut{P_n(\Sigma_{g,p})}$ defined  as follows. For  the case of  $\Phi$, 
given $h\in [h]\in M^{\ast}_n(\Sigma_{g,p})$  and $\{\alpha_1,\cdots,\alpha_n\}$  a set of paths between elements of $z$ which is  a representative of an element of $B_n(\Sigma_{g,p})$  define  $h_{n\#}[\{\alpha_1,\cdots,\alpha_n\}]$ as the class determined by the set of paths
 $\{h\circ \alpha_1,\cdots,h\circ \alpha_n\}$. It is straightforward  to see that this map is well defined, it is a homomorphism, and $\Phi$ is also a homomorphism. For  the case of $\Phi_0$  we have a similar definition, where we only stress the point that needs to be adapted in  the description above.  Given $h$ and a representative 
 $(\alpha_1,\cdots,\alpha_n)$  of an element  of $P_n(\Sigma_{g,p}, (z_1,\cdots,z_n))$ 
 define  $h_{n\#}[(\alpha_1,\cdots,\alpha_n)]$ as the class determined by the ordered  sequence of loops 
 $(h\circ \alpha_{\sigma_h^{-1}(1)},\cdots,h\circ \alpha_{\sigma_h^{-1}(n)})$ having base point  $(z_1,\cdots, z_n)$. 
 The rest is similar.\\
 Let $\Sigma_{g,p}$ be an orientable surface of genus $g$ with $p\geq 0$ points removed. 
The group $\aut{P_n(\Sigma_{g,p})}$ has been  studied by many authors and it was proved that there exists an isomorphism  
\begin{equation}\label{autpn} 
M^{\ast}_n(\Sigma_{g,p}) \simeq \aut{P_n(\Sigma_{g,p})}
\end{equation} 
for (not necessarily closed) orientable surfaces   with Euler characteristic $\chi(\Sigma_{g,p})<-1$ 
by  \cite[Theorem~1.3]{A}. 
The isomorphism \eqref{autpn} is the map $\Phi_0$ as described above, 
for more details see 
\cite{A}. \\
So the following cases, for  orientable surfaces, are  not  covered by the results above:
\begin{itemize}
    \item[I)] for $g=1$ and $p=0,1$ (i.e.\ the surfaces of $\mathcal{F}_3$); 
    \item[II)] $g=0$,      $p=0, 1, 2, 3$. (i.e.\ the surfaces of $\mathcal{F}_1$ and the surface $\Sigma_{0,3}$).
\end{itemize}

Now we prove that the short exact sequence of Theorem~\ref{goldb} is characteristic with respect to automorphisms. 
This will be used to study the R$_{\infty}$-property of the groups $P_n(\Sigma_{g,p})$ and $B_n(\Sigma_{g,p})$ in the next section.

\begin{lem}\label{lem:gold1} 
Let $\Sigma_{g,p}$ be a surface of the family ${\mathcal F}_2$, being different from the sphere minus three points.
 Let $N$ denote the normal subgroup generated by the image of the Artin pure braids via the inclusion of the disc $D$ into $\Sigma_{g,p}$, $i\colon D \hookrightarrow \Sigma_{g,p}$, and $n\geq 2$. 
If $\varphi\colon P_n(\Sigma_{g,p}) \to P_n(\Sigma_{g,p})$ is an automorphism then $\varphi(N)\subset N$. Therefore the short exact sequence given in \req{sesg} is characteristic. 
\end{lem}  

\begin{proof} 
Let $\varphi$ be an automorphism of $P_n(\Sigma_{g,p})$. 
To show that $\varphi(N)\subset N$ it is enough to prove that $\varphi(i_\#(P_n(D)))\subset N$, since $N$ is the normal subgroup generated by the images of the Artin pure braids via the inclusion $i\colon D \hookrightarrow \Sigma_{g,p}$ of the disc $D$ into $\Sigma_{g,p}$.
From the hypothesis of this lemma the Euler characteristic $\chi(\Sigma_{g,p})<-1$, and so there is a homeomorphism $h\colon \Sigma_{g,p}\to \Sigma_{g,p}$ such that $\varphi$ is induced by $h$, i.e.\ $\varphi=\Phi_0([h])\in \aut{P_n(\Sigma_{g,p})}$ (by \eqref{autpn} and the discussion above).

Let $\alpha\in P_n(D)$. We shall prove that $\varphi(i_\#(\alpha))\in N$. The inclusion $i\colon  D \hookrightarrow \Sigma_{g,p}$ induces an inclusion $\widehat{i}\colon  F_n(D) \hookrightarrow F_n(\Sigma_{g,p})$. 
Geometrically, since $P_n(D)=\pi_1(F_n(D),(z_1,\ldots,z_n))$, we have that $\alpha$ has a representative $\widehat{\alpha}=(\alpha_1,\alpha_2,\ldots,\alpha_n)$ where $\alpha_j\colon [0,1]\to D$ is a  loop in $D$ with base point $z_j\in D$. 
By definition, $\Phi_0([h])(i_\#(\alpha))$ is the 
homotopy class of the loop $(h\circ i\circ \alpha_{\sigma_h^{-1}(1)},\cdots,h\circ i\circ \alpha_{\sigma_h^{-1}(n)})$ based at $(z_1,\cdots, z_n)$ in $\Sigma_{g,p}$.

Now let $\widehat{\psi}\colon F_n(\Sigma_{g,p}) \to \Pi_{i=1}^n(\Sigma_{g,p}^n)$ be the inclusion.
Since $h\circ i\circ \alpha_{\sigma_h^{-1}(j)}$ is a  loop based in $z_j$ inside $h(D)\subset \Sigma_{g,p}$, for every $1\leq j\leq n$, and since $\eval{h}_D\colon D\to h(D)$ is a homeomorphism, we get that $\widehat{\psi}((h\circ i \circ \alpha_{\sigma_h^{-1}(1)},\cdots,h\circ i \circ \alpha_{\sigma_h^{-1}(n)}))$ is homotopic to the $n$-tuple of constant maps $(c_{z_1}, \ldots, c_{z_n})$. \\
Algebraically, the induced map $\widehat{\psi}_{\#}$ on the level of the fundamental groups corresponds to the surjective map $P_n(\Sigma_{g,p}) \to \Pi_{i=1}^n(\pi_1(\Sigma_{g,p}))$ in the short exact sequence given in \req{sesg}. 
Therefore, $\widehat{\psi}_{\#}(\Phi_0([h])(i_\#\alpha)) = 1 \in \Pi_{i=1}^n(\pi_1(\Sigma_{g,p}))$. 
Hence, $\Phi_0([h])(i_\#\alpha)=\varphi(i_\#\alpha)\in N$.
\end{proof}

\section{The R$_{\infty}$-property for $P_n(\Sigma_{g,p})$ and $B_n(\Sigma_{g,p})$}\label{sec:main} 

In this section we prove \reth{purerinfty} and \reth{totalrinfty}.

\subsection{The Fadell-Neuwirth short exact sequence}\label{subsec:fadell}

Let $S$ be a connected surface and let $n\in \mathbb{N}$.
If $m\geq 1$, the map $p\colon F_{n+m}(S)\to F_n(S)$, of the configuration space $F_{n+m}(S)$ onto $F_n(S)$, defined by $p(x_1,\ldots,x_n,\ldots,x_{n+m}) = (x_1,\ldots,x_n)$ induces a homomorphism $p_{\ast}\colon P_{n+m}(S)\to P_n(S)$.
The homomorphism $p_{\ast}$ geometrically ``forgets'' the last $m$ strings.
If $S$ is without boundary, Fadell and Neuwirth showed that $p$ is a locally-trivial fibration \cite[Theorem~1]{FaN}, with fibre $F_m(S\setminus \{ x_1,\ldots,x_n \})$ over the point $(x_1,\ldots,x_n)$, which we consider to be a subspace of the total space via the map $i\colon F_m(S\setminus \{x_1,\ldots,x_n\})\to F_{n+m}(S)$ defined by $i((y_1,\ldots,y_m)) = (x_1,\ldots,x_n,y_1,\ldots,y_m)$.
Applying the associated long exact sequence in homotopy to this fibration, we obtain the Fadell-Neuwirth short exact sequence of pure braid groups:
\begin{equation}
1\to P_{m}(S\setminus \{x_1,\ldots,x_n\}) \stackrel{i_{\ast}}{\longrightarrow} P_{n+m}(S) \stackrel{p_{\ast}}{\longrightarrow} P_n(S) \to 1
\end{equation}
where $n\geq 3$ if $S$ is the sphere \cite{Fa, FvB}, $n\geq 2$ if $S$ is the real projective plane \cite{vB}, and $n\geq 1$ otherwise \cite{FaN}, and $i_{\ast}$ is the homomorphism induced by the map $i$.
This sequence has been widely studied.
For instance, one question studied by many authors during several years was the splitting problem for surface pure braid groups, and it was completely solved, see \cite{GG2} for more details, in particular its Theorem~2. Additional information on this sequence may be seen in \cite[Section~3.1]{GPi}.

We are interested in (quotients by) the center of some surface braid groups. 
Seemingly the content of Proposition~\ref{th:annulus}  related to the braid groups over the annulus is well known for the experts in braid theory, however to the best of our knowledge there is no proof in the literature of it. Hence, for the sake of completeness, we provide a proof here. 
The following information will be useful: From \cite[Proposition~4.1]{PR}, for all $n\geq 1$, the center of the group $B_n(\Sigma_{0,2})$ is isomorphic to $\Z$ generated by $\alpha_n \in P_n(\Sigma_{0,2})$. See \cite[Figure~4.1]{PR} for a geometric description of this element.

\begin{prop}\label{th:annulus}  Let $n\geq 1$.\\
Then   $Z(B_n(\Sigma_{0,2})) = Z(P_n(\Sigma_{0,2}))$ and 
$P_{n+1} (\Sigma_{0,2})/Z(P_{n+1} (\Sigma_{0,2}))\cong P_n(\Sigma_{0,3})$.
\end{prop}
\begin{proof}
Let $n=1$, then $P_1(\Sigma_{0,2}) = B_1(\Sigma_{0,2})\cong \Z$ and so 
$Z(P_1(\Sigma_{0,2})) = Z(B_1(\Sigma_{0,2}))$. To study the situation with more than one string, we consider the
    short exact sequence induced from the Fadell-Neuwirth fibration
    $$
    1\to P_{n}(\Sigma_{0,3}) \to P_{n+1}(\Sigma_{0,2}) \stackrel{p_{\ast}}{\longrightarrow} P_{1}(\Sigma_{0,2}) \to 1
    $$
where $p_{\ast}$ geometrically forgets the last $n$ strings. 
We note that $P_{1}(\Sigma_{0,2})=\pi_{1}(\Sigma_{0,2})\cong \Z$ is generated by $\alpha_1\in P_{1}(\Sigma_{0,2})$. 
Consider the section of $p_{\ast}$ that sends $\alpha_1\in P_{1}(\Sigma_{0,2})$ onto $\alpha_{n+1}\in P_{n+1}(\Sigma_{0,2})$. 
Since $\alpha_{n+1} \in Z(B_{n+1}(\Sigma_{0,2}))$, $\alpha_{n+1} $ also commutes with all elements from 
$P_{n+1}(\Sigma_{0,2})$ and so we find that 
$$
P_{n+1}(\Sigma_{0,2}) \cong P_{n}(\Sigma_{0,3}) \oplus \Z,
$$
where $\Z$ is the subgroup of $P_{n+1}(\Sigma_{0,2})$ generated by $\alpha_{n+1}$. As $Z(P_n(\Sigma_{0,3}))=1$ \cite[Proposition~1.6]{PR}, it now readily follows that 
$Z(P_{n+1}(\Sigma_{0,2})) = Z(B_{n+1}(\Sigma_{0,2}))$ and hence
$P_{n+1} (\Sigma_{0,2})/Z(P_{n+1} (\Sigma_{0,2}))\cong P_n(\Sigma_{0,3})$.
\end{proof}

The information of some surface braid groups in the following remarks will be useful in this section. 

\begin{rems}\label{rems:sbg}
Suppose $n\geq 1$. 

\begin{enumerate}
    \item Braid groups with few strings over the sphere  
    are finite:\\ 
   $B_1(\St[2])$, $P_1(\St[2])$ and $P_2(\St[2])$ are trivial groups,  $B_2(\St[2])\cong \Z_2$, $B_3(\St[2])$ is isomorphic to $\Z_3\rtimes \Z_4$ with non-trivial action and $P_3(\St[2])\cong \Z_2$, see \cite{FvB} and also \cite[Section~4]{GPi}.

\item If $S=\St[2]$ is the sphere, then $P_{n+3}(\St[2])\cong P_n(\Sigma_{0,3})\oplus \Z_2$ (see \cite[Theorem~4]{GG1}).
We remark here that $Z(P_{n+3}(\St[2]))=\Z_2$ because $Z(P_n(\Sigma_{0,3}))=1$ \cite[Proposition~1.6]{PR}.  Hence $P_{n+3}(\St[2])/Z(P_{n+3}(\St[2]))\cong P_n(\Sigma_{0,3})$.

\item Suppose $S=D$ is the disc. It is an immediate consequence of the classical Artin presentation of $P_2(D)$ and $B_2(D)$ that they are isomorphic to $\Z$, see \cite{A2}.
Let $n\geq 3$. There is a decomposition $P_{n}(D)\cong P_{n-2}(\Sigma_{0,3})\oplus \Z$ 
that follows from the splitting of the Fadell-Neuwirth short exact sequence (see \cite[Theorem~4]{GG1}). Using 
\cite[Proposition~1.6]{PR} again, we find that  $P_{n-2}(\Sigma_{0,3})\cong P_n(D)/Z(P_n(D))$.
\end{enumerate}

\end{rems}

\subsection{The proof of \reth{purerinfty} and \reth{totalrinfty}}

Recall that we split the orientable surfaces of finite type into three families, as follows:
\begin{itemize}
	\item[$\mathcal{F}_1$:] The punctured sphere  $\mathbb{S}^2$ with $p$ points removed for $p=0,1,2$.  

\item[$\mathcal{F}_2$:] 
\begin{itemize}
	\item[a)] Orientable closed surfaces different from $\mathbb{S}^2$, $T^2$. 
	\item[b)] Orientable punctured surfaces $\Sigma_{g,p}$ where $g$ is the genus and $p$ is the number of punctures in the closed surface $\Sigma_g$, for: 
\begin{itemize}
	\item[i)] $g=0$ and $p\geq 3$,
	\item[ii)] $g=1$ and $p\geq 2$,
	\item[iii)] $g\geq 2$ and $p\geq 1$. 
\end{itemize}
	\end{itemize}
	
\item[$ \mathcal{F}_3$:] The torus and the once punctured torus.

\end{itemize}

We will show the results for the families $\mathcal{F}_1$ and $\mathcal{F}_2$, using two different arguments.    
The case $\mathcal{F}_3$ is work in progress.

We shall use the following technique
in order to prove that the pure braid groups of surfaces $\Sigma_{g,p}$ (closed or punctured) from $\mathcal{F}_2$ have the R$_{\infty}$-property whenever $\pi_1(\Sigma_{g,p})$ has the R$_{\infty}$-property. 
If $\alpha$ is an automorphism of a group $G$ and $N$ is a normal subgroup of $N$ with $\alpha(N)=N$ (e.g.\ when $N$ is a characteristic subgroup of $G$) then $\alpha$ induces an automorphism $\bar\alpha$ of $G/N$. It is easy to see that $R(\alpha) \geq R(\bar\alpha)$, so if $R(\bar\alpha)= \infty$, then also $R(\alpha)=\infty$. 
For all $\Sigma_{g,p}$ of family $\mathcal{F}_2$, it holds that $\pi_1(\Sigma_{g,p})$ has the R$_{\infty}$-property.

This does not longer hold for the surfaces of family $\mathcal{F}_1$ and we will solve this case using a different argument.

We will state the  result for the family $\mathcal{F}_2$. 	To prove the result we will make use of \relem{gold1} for 
all surface $\Sigma_{g,p}$ in $\mathcal{F}_2$, except for the case in which $g=0$ and $p=3$, since the validity of 
the equation  \eqref{autpn} is not covered by An \cite{A}, because $\chi(\Sigma_{g,p}) = -1$. So, for this special case  of the family 
$\mathcal{F}_2$  we shall use another approach.

  \begin{prop}\label{mainI} For any surface $\Sigma_{g,p}\in \mathcal{F}_2$ we have that  $P_n(\Sigma_{g,p})$ has the R$_{\infty}$-property for $n\geq 1$.
  \end{prop}
  \begin{proof} For  $n=1$ the group $P_1(\Sigma_{g,p})=\pi_1(\Sigma_{g,p})$ is the fundamental group of the surface. For every surface 
in the family $\mathcal{F}_2$ its fundamental group has the R$_{\infty}$-property. 
  Indeed, for all $\Sigma_{g,p} \in \mathcal{F}_2$, it holds that $\chi(\Sigma_{g,p})<0$ and so $\pi_1(\Sigma_{g,p})$ is non-elementary hyperbolic, hence the result follows from \cite{LL}, see also \cite{F}.

Let $\Sigma_{0,3}$ be the pantalon. 
Recall that  
$P_n(\Sigma_{0,3})\cong P_{n+2}(D)/Z(P_{n+2}(D))$  (see  \rerems{sbg} in \resubsec{fadell}). 
 In \cite{DGO} the group  $P_{n+2}(D)/Z(P_{n+2}(D))$ 
was denoted by  $\overline P_{n+2}$. Moreover,  in order to prove the main result of \cite{DGO} it was shown that  this group has the R$_{\infty}$-property (see the last two lines of page 17 of \cite{DGO}). 

Let $n\geq 2$ and let $\Sigma_{g,p}\neq \Sigma_{0,3}$ be a surface from ${\mathcal F}_2$. From \relem{gold1}  the short exact sequence given in \req{sesg} is characteristic, then 
we conclude the result in these cases since $\Pi_{i=1}^n(\pi_1(S))$ has the R$_{\infty}$-property by \cite[Corollary~4.5]{S}. 
\end{proof}

Now we move to the surfaces of the family $\mathcal{F}_1$.
We notice that, for $n\geq 2$, the result that $P_n(\Sigma_{0,2})$ has the R$_{\infty}$-property  was already proved in \cite{CS}, using a different technique from the ones used  here.

\begin{prop}\label{mainII}  Let  $\Sigma_{g,p}$ be a surface of the family $\mathcal{F}_1$. Then $P_n(\Sigma_{g,p})$ has the R$_{\infty}$-property if and only if one of the following cases holds:
\begin{enumerate}   
\item  $\Sigma_{g,p}=\mathbb{S}^2$ and $n\geq 4$, 
\item  $\Sigma_{g,p}=\mathbb{S}^2\setminus\{x_1\}$ and $n\geq 3$,
\item  $\Sigma_{g,p}=\mathbb{S}^2\setminus\{x_1, x_2\}$ and $n\geq 2$.
\end{enumerate}

\end{prop}

\begin{proof}  The proof is case by case, but first we deal with the exceptional cases. 
It is well known (see \rerems{sbg} in \resubsec{fadell}) that $P_1(\mathbb{S}^2)=P_2(\mathbb{S}^2)=\{1\}$, $P_3(\mathbb{S}^2)=\Z_2$, $P_1(\mathbb{S}^2\setminus\{x_1\})=\{1\}$, $P_2(\mathbb{S}^2\setminus\{x_1\})=\Z$, $P_1(\mathbb{S}^2\setminus\{x_1, x_2\})=\Z$
and all these groups do not have the R$_{\infty}$-property.

For $\Sigma_{g,p}=\mathbb{S}^2$ and $n\geq4$ we have that $P_n(\mathbb{S}^2)/Z(P_n(\mathbb{S}^2))\cong P_{n-3}(\mathbb{S}^2\setminus\{x_1,x_2,x_3\})$ (see \rerems{sbg} in \resubsec{fadell}). 
From  Proposition~\ref{mainI},    $P_{n-3}(\mathbb{S}^2\setminus\{x_1,x_2,x_3\})$ has the R$_{\infty}$-property for $n-3\geq 1$,  
then the result follows since $Z(P_n(\mathbb{S}^2))$ is a characteristic subgroup of $P_n(\mathbb{S}^2)$.
Let  $\Sigma_{g,p}=\mathbb{S}^2\setminus\{x_1\}$, then $P_n(\Sigma_{g,p})$ is the pure Artin braid group and the result follows from the main result of \cite{DGO}.

  For   $\Sigma_{g,p}=\mathbb{S}^2\setminus\{x_1, x_2\}$ and $n\geq 2$ we know 
  $P_n(\mathbb{S}^2\setminus\{x_1, x_2\})/Z(P_n(\mathbb{S}^2\setminus\{x_1, x_2\})) \cong P_{n-1}(\mathbb{S}^2\setminus\{x_1,x_2,x_3\})$, by Proposition~\ref{th:annulus}. Again by using the fact that the centre is a characteristic subgroup and using Proposition~\ref{mainI} for $\Sigma_{0,3}$ the result follows. 
    \end{proof}

From the discussion above we may prove the main result about the R$_{\infty}$-property for surface pure braid group $P_n(\Sigma_{g,p})$.

\begin{proof}[Proof of \reth{purerinfty}]  The result follows immediately from Propositions \ref{mainI} and \ref{mainII}.
\end{proof}

Now we consider the groups $B_n(\Sigma_{g,p})$ for $n\geq 2$, and we will make use of the short exact sequence
 $$1\to P_n(\Sigma_{g,p})\to B_n(\Sigma_{g,p})\to S_n \to 1.$$ 
We recall that by  \cite[Theorem 1.5]{A} this short exact sequence is characteristic unless we are the case that 
$\Sigma_{g,p}=\Sigma_{0, 2}$ and $n=2$. In the  cases where the   short exact sequence is characteristic  we will use the following result given by   \cite[Lemma~6]{MS}:
 Let $1\to A\to B\to C\to 1$ be a characteristic short exact sequence with respect to automorphisms of $B$  such that $C$ is finite and $A$ has the R$_{\infty}$-property.  Then    $B$ has the R$_{\infty}$-property. 
To analyse the case $B_2(\Sigma_{0,2})$ we will use a different approach.

Now we prove the main result for the groups $B_n(\Sigma_{g,p})$. 
We notice that, for $n\geq 2$, the result that $B_n(\Sigma_{g,p})$ has the R$_{\infty}$-property  was already proved, using  different techniques from the ones here, in \cite{FG} for $\Sigma_{g,p}=\mathbb{S}^2$  and $\Sigma_{g,p}=D^2$,
and in \cite{CS} for $\Sigma_{g,p}=\Sigma_{0,2}$ and $\Sigma_{g,p}=\Sigma_{0,3}$.

\begin{proof}[Proof of \reth{totalrinfty}]
Recall that for any surface $S_{g,p}$ there exists a short exact sequence
$$
1\to P_n(S_{g,p}) \to B_n(S_{g,p}) \to S_n \to 1.
$$
 By  \cite[Theorem 1.5]{A} this short exact sequence is characteristic as long as $\Sigma_{g,p}\ne \Sigma_{0, 2}$ or $n>2$. 
 For the if part let  $P_n(\Sigma_{g,p})$ has the R$_{\infty}$-property. From   \cite[Lemma~6]{MS}   the result follows for $B_n(\Sigma_{g,p})$. 
The only remaining case is when $\Sigma_{g,p}=\Sigma_{0, 2}$ and $n=2$. But from \cite[Proposition~2.1~(2)]{CrPa}, $B_2(\Sigma_{0, 2})$ is isomorphic 
to  $F_2(x,y)\rtimes_{\theta} \mathbb{Z}$. The semi-direct product is determined  by the  action given by the automorphism $\theta(1)(x)=y$,
$\theta(1)(y)=y^{-1}xy$.  Now it follows from  \cite[Theorem 4.4]{FGW} that  $B_2(\Sigma_{0, 2})$ has the R$_{\infty}$-property (independent of the action).

For the only if part, observe that the only cases where $B_n(\Sigma_{g,p})$ ($n>1$) does not have the R$_{\infty}$-property are:
\begin{itemize}
	\item[a)]   If $\Sigma_{g,p}=\mathbb{S}^2$ for  $2\leq n\leq 3$ because the groups are finite;
	\item[b)]  If $\Sigma_{g,p}=\mathbb{S}^2\setminus\{x_1\}$ for  $n=2$ because the group is isomorphic to $\Z$.
\end{itemize}
 In both cases the groups $P_n(\Sigma_{g,p})$ also does not have 
the R$_{\infty}$-property, see Theorem~\ref{th:purerinfty}. So the result follows.
\end{proof}


\begin{thebibliography}{GGO2}

{\small



 \bibitem[A]{A} B.~H.~An,  Automorphisms of braid groups on orientable surfaces, \emph{J. Knot Theory Ramifications} \text{25} (2016), no. 5, 1650022, 32 pp.


\bibitem[A1]{A1} E.~Artin, Theorie der Z\"opfe, \emph{Abh.\ Math.\ Sem.\ Univ.\ Hamburg} \textbf{4} (1925), 47--72.

\bibitem[A2]{A2} E.~Artin, Theory of braids, \emph{Ann.\ Math.} \textbf{48} (1947), 101--126.


\bibitem[B1]{B1} P.~Bellingeri, On presentations of surface braid groups, \emph{J. Algebra} \textbf{274} (2004), 543--563.

\bibitem[B2]{B2} P.~Bellingeri, Private communication.


\bibitem[BGG]{BGG} P.~Bellingeri, S.~Gervais and J.~Guaschi, Lower central series of Artin-Tits and surface braid groups, \emph{J. Algebra} \textbf{319} (2008), 1409--1427.

\bibitem[Bi]{Bi} J.~S.~Birman, On braid groups, \emph{Comm. Pure Appl. Math.} \textbf{22} (1969), 41--72.



\bibitem[vB]{vB} J.~Van~Buskirk, Braid groups of compact 2-manifolds with elements of finite order, \emph{Trans. Amer. Math. Soc.} \textbf{122} (1966),  81--97.

\bibitem[CS]{CS} M.~Calvez and I.~Soroko, Property R$_{\infty}$ for some spherical and affine Artin-Tits groups, \emph{J. Group Theory} \textbf{25} (2022), no. 6, 1045--1054.

\bibitem[CrPa]{CrPa} J. Crisp  and L. Paris, Artin groups of type $B$ and $D$, Preprint, Laboratoire de  Topologie-Universit\'e de Bourgogne (2002).

\bibitem[DGO]{DGO} K.~Dekimpe, D.~L.~Gon\c calves and O.~Ocampo, The $R_{\infty}$ property for pure Artin braid groups, \emph{Monatshefte f\"ur Mathematik} \textbf{195} (2021), (1), 15--33. 

\bibitem[DGO2]{DGO2} K.~Dekimpe, D.~L.~Gon\c calves and O.~Ocampo, Characteristic subgroups and the R$_\infty$-property for virtual braid groups, \emph{J. Algebra} \textbf{663} (2025), 20--47.  

\bibitem[D]{D} R.~Diniz, Grupos de tran\c cas de superf\'icies finitamente perfuradas e grupos cristalogr\'aficos, PhD thesis, Universidade Federal de S\~ao Carlos, 2020.

\bibitem[FvB]{FvB} E.~Fadell and J.~Van~Buskirk, The braid groups of $E^2$ and $S^2$, \emph{Duke Math. J.} \textbf{29} (1962),  243--257.

\bibitem[FaN]{FaN} E.~Fadell and L.~Neuwirth, Configuration spaces, Math. Scand. 10 (1962), 111--118.

\bibitem[Fa]{Fa} E.~Fadell, Homotopy groups of configuration spaces and the string problem of Dirac, \emph{Duke Math. J.} \textbf{29} (1962), 231--242.

\bibitem[Fe]{F} A.~Fel'shtyn, The Reidemeister number of any automorphism of a Gromov hyperbolic group   is infinite, \emph{J. Math. Sci.} \textbf{119} (2004), no. 1, 117--123

\bibitem[FG1]{FG1} A.~Fel'shtyn and D.~L.~Gon\c{c}alves, The Reidemeister number of any automorphism of a Baumslag--Solitar group is infinite, \emph{Geometry and Dynamics of Groups and Spaces}. Prog Math 265, 286--306 (2008) Birkh\"auser.

\bibitem[FG2]{FG} A.~Fel'shtyn and D.~L.~Gon\c{c}alves, Twisted conjugacy classes in symplectic groups, mapping class groups and braid groups, \emph{Geom.\ Dedicata} \textbf{146} (2010), 211--223.


\bibitem[FGW]{FGW} A.~Fel'shtyn, D.~L.~Gon\c{c}alves and P.~Wong,   Twisted conjugacy classes for polyfree groups, 
\emph{Comm. Algebra} \textbf{42}  (2014), no. 1, 130--138. 

\bibitem[FH]{FH} A.~Fel'shtyn and R.~Hill, The Reidemeister zeta function with applications to Nielsen theory and a connection with Reidemeister torsion, \emph{K-Theory}  \textbf{8}  (1994),  no. 4, 367--393.
 
\bibitem[FN]{FN}  A.~Fel'shtyn and T.~Nasybullov, The $R_\infty$ and $S_\infty$ properties for linear algebraic groups, \emph{J. Group Theory}  \textbf{19}  (2016),  no. 5, 901--921.

\bibitem[FoN]{FoN} R.~H.~Fox and L.~Neuwirth, The braid groups, \emph{Math.\ Scandinavica} \textbf{10} (1962), 119--126.

\bibitem[G]{G} C.~H.~Goldberg, An exact sequence of braid groups, \emph{Math. Scand.} \textbf{33} (1973), 69--82. 

\bibitem[GG1]{GG1} D.~L.~Gon\c{c}alves and  J.~Guaschi, The roots of the full twist for surface braid groups, \emph{Math. Proc. Camb. Phil. Soc.} \textbf{137} (2004), 307--320. 


\bibitem[GG2]{GG2} D.~L.~Gon\c{c}alves and  J.~Guaschi, Braid groups of non-orientable surfaces and the Fadell-Neuwirth short exact sequence, \emph{J. Pure Appl. Algebra} \textbf{214} (2010),  667--677.

\bibitem[GG3]{GG3} D.~L.~Gon\c{c}alves and  J.~Guaschi, Inclusion of configuration spaces in Cartesian products, and the virtual cohomological dimension of the braid groups of $\mathbb{S}^2$ and $\mathbb{R}P^2$, \emph{Pac. J. Math.} \textbf{287} (2017),  71--99.

\bibitem[GW]{GW} D.~Gon\c{c}alves and P.~Wong,  Twisted conjugacy classes in nilpotent groups,    \emph{J. Reine Angew. Math.} \textbf{633}  (2009), 11--27.

\bibitem[GJ-P]{GPi} J.~Guaschi, D.~Juan-Pineda, A survey of surface braid groups and the lower algebraic K-theory of their group rings, from: ``Handbook of group actions, Vol. II'', (L Ji, A Papadopoulos, S-T Yau, editors), Adv. Lect. Math. 32, Int. Press, Somerville, MA (2015), 23--75.

\bibitem[L]{La} S. Lambropoulou, Braid structures in knot complements, handlebodies and 3-manifolds, in: Proceedings of Knots in Hellas '98, in: Series of Knots and Everything, vol. 24, World Scientific Press, 2000, pp. 274--289. 


\bibitem[LL1]{LL} G.~Levitt and M.~Lustig, Most automorphisms of a hyperbolic group have very simple dynamics,\emph{ Ann. Sci. \'Ecole Norm. Sup. (4)}  \textbf{33}  (2000),  no. 4, 507--517.

\bibitem[LL2]{L} G.~Levitt and M.~Lustig, On the automorphism group of generalized Baumslag-Solitar groups.
\emph{Geom. Topol. } \textbf{11}  (2007), 473--515.

\bibitem[MS]{MS} T. Mubeena  and P. Sankaran, Twisted conjugacy and quasi-isometric rigidity of irreducible   lattices in semisimple  Lie groups, \emph{Indian J. Pure Appl. Math.}, 50(2)  (2019),  403-412.


\bibitem[NNS]{NNS} T.~Naik, N.~Nanda and M.~Singh, Structure and automorphisms of pure virtual twin groups, \emph{Monatsh. Math.} \textbf{202} (3)  (2023), 555--582.

\bibitem[NAS]{N} T.~Nasybullov,  Twisted conjugacy classes in unitriangular groups, \emph{J. Group Theory}  \textbf{22}  (2019),  no. 2, 253--266.



\bibitem[PR]{PR} L.~Paris and D.~Rolfsen, Geometric subgroups of surface braid groups, \emph{Ann. Inst. Fourier} \textbf{49}
(1999), 417--472.



\bibitem[S]{S} P.~Senden, Twisted conjugacy in direct products of groups, \emph{Comm. Algebra} \textbf{49} (2021), no. 12, 5402--5422.

\bibitem[TW]{TW} J.~Taback and P.~ Wong,  Twisted conjugacy and quasi-isometry invariance for generalized solvable Baumslag-Solitar groups, \emph{J. Lond. Math. Soc. (2)}  \textbf{75}  (2007),  no. 3, 705--717

\bibitem[TRO]{T} E.~Troitsky,  Reidemeister classes in lamplighter-type groups, \emph{Comm. Algebra}  \textbf{47}  (2019),  no. 4, 1731--1741.


\bibitem[Z]{Z} O.~Zariski, The topological discriminant group of a Riemann surface of genus $p$, \emph{Amer.\ J.\ Math.} \textbf{59} (1937), 335--358.


}


\end{thebibliography}
\end{document}